\documentclass{amsart}
\def\LaTeX{L\kern -.36em\raise .3ex\hbox{\sc a}\kern -.15em T\kern -.1667em%
\lower .7ex\hbox{E}\kern -.125em X}

\newtheorem{thm}{Theorem}
\newtheorem{cor}{Corollary}

\newtheorem{pr}{Proposition}
\newtheorem{df}{Definition}

\newtheorem{ex}{Example}

\begin{document}
\thanks{Research of the author supported by Institute of Mathematics, Polish Academy of Sciences.}
\author{Wojciech Domitrz}
\address{Warsaw University of Technology,
Faculty of Mathematics and Information Science\\
Plac Politechniki 1, 00-661 Warsaw,
Poland\\
domitrz@mini.pw.edu.pl}
\date{}

\title[Reductions of locally conformal symplectic structures ]{Reductions \\ of locally conformal symplectic structures \\and  de Rham cohomology \\tangent to a foliation.}
\maketitle

\section{Introduction.}

Let $M$ be a smooth even-dimensional manifold, $\dim M=2n>2$. Let
$\Omega$ be a smooth nondegenerate $2$-form on $M$. If there
exists an open cover $\left\{ U_a: a\in A \right\}$ of $M$ and
smooth positive functions $f_a$ on $U_a$ such that
\begin{equation}
\label{def1} \Omega_a=f_a\Omega|_{U_a}
\end{equation}
is a symplectic form on $U_a$ for $a\in A$ then $\Omega$ is called
a {\bf locally conformal symplectic form}.

Equivalently (see \cite{L}) $\Omega$ satisfies the following
condition:
\begin{equation}
\label{def2}
 d\Omega=\omega\wedge \Omega,
\end{equation}
where $\omega$ is a closed $1$-form. $\omega$ is uniquely
determined by $\Omega$ and is called {\bf the Lee form of}
$\Omega$. $(M,\Omega,\omega)$ is called a {\bf locally conformal
symplectic manifold}.

If $\Omega$ satisfies (\ref{def1}) then $\omega|_{U_a}=d(\ln f_a)$
for all $a\in A$. If  $f_a$ is constant for all $a\in A$ then
$\Omega$ is a symplectic form on $M$. The Lee form of the
symplectic form is obviously zero.

Locally conformal symplectic manifolds are generalized phase
spaces of Hamiltonian dynamical systems since the form of the
Hamiltonian equations is preserved by homothetic canonical
transformations \cite{Vaisman}.

Two locally conformal symplectic forms $\Omega_1$ and $\Omega_2$
on $M$ are {\bf conformally equivalent} if $\Omega_2=f \Omega_1$
for some smooth positive function $f$ on $M$. A conformal
equivalence class of locally conformal symplectic forms on $M$ is
a {\bf locally conformal symplectic structure} on $M$ (\cite{Ba}).

Let $Q$ be a smooth submanifold of $M$. Let
$\iota:Q\hookrightarrow M$ denote the standard inclusion. We say
that two locally conformal symplectic forms $\Omega_1$ and
$\Omega_2$ on $M$ are {\bf conformally equivalent on $Q$} if
$\iota^{\ast}\Omega_2=f \iota^{\ast}\Omega_1$ for some smooth
positive function $f$ on $Q$.

Clearly the Lee form of a locally conformal symplectic form is
exact if and only if  $\Omega$ is conformally equivalent to a
symplectic form \cite{Vaisman}. Then the locally conformal
symplectic structure is {\bf globally conformal symplectic}.

Locally conformal symplectic forms were introduced by Lee
\cite{L}. They have been  intensively studied in \cite{Vaisman},
\cite{H-R1}, \cite{H-R2}, \cite{H-R3}, \cite{Ba}.

In \cite{Wojtkowski} the symmetry of the Lyapunov spectrum in
locally conformal Hamiltonian systems is studied. It was shown
that Gaussian isokinetic dynamics, Nos\'{e}-Hoovers dynamics and
other systems can be treated as locally conformal Hamiltonian
systems. A kind of reduction was applied to obtain these results.

In  \cite{H-R3} (see Section 3) a reduction procedure of a locally
conformal symplectic form is defined using the general definition
of reduction (see \cite{MR}). But the conditions for the reduction
of locally conformal symplectic form are very restrictive (see
Proposition 1 in \cite{H-R3} and Proposition \ref{tw-ex} in
Section \ref{form}). There are local obstructions, a locally
conformal symplectic form on a germ of a generic smooth
hypersurface cannot be reduced using this procedure (see Example
\ref{p1}). The procedure of the reduction of a locally conformal
symplectic form has no application for the reduction of systems
with symmetry defined in Section 5 of \cite{H-R3}.

In this paper we show different approach to this problem.  We
propose to reduce a locally conformal symplectic {\em structure}
(the conformal equivalence class of a locally conformal symplectic
form) instead of a locally conformal symplectic form. This
procedure of reduction can be applied to much wider class of
submanifolds. There are no local obstructions for this procedure.
But there are global obstructions. We find a necessary and
sufficient condition when this reduction holds in terms of the
special kind of de Rham cohomology class (tangent to the
characteristic foliation) of the Lee form.

\section{De Rham cohomology tangent to a foliation.}

Let $Q$ be a smooth manifold and let $\mathcal F$ be a foliation
in $Q$. We denote by $\Omega ^p(Q)$ the space of differential
$p$-forms on $Q$. By $\Omega ^p(Q,\mathcal F)$ we denote the space
of $p$-forms $\omega $ satisfying the following condition:
$$\omega|_q
(v_1,\ldots ,v_p)=0$$  for any $q\in Q$ and for any vectors
$v_1,\ldots ,v_p$ tangent to the foliation $\mathcal F$ at $q$. It
means that $\omega \in \Omega ^p(Q,\mathcal F)$ if and only if
$\iota^{\ast}_q\omega =0$ for any $q\in Q$, where $i_q:{\mathcal
F}_q\hookrightarrow Q$ is the standard inclusion  of the leaf
${\mathcal F}_q$ of the foliation ${\mathcal F}$ into $Q$.

$\Omega ^p(Q,\mathcal F)$ is a subcomplex of the de Rham complex
$\Omega^{\ast}(Q)$. It follows from the relation
$\iota^{\ast}_qd\omega =d\iota^{\ast}_q\omega$.

 We define the factor space
$$\Omega^p({\mathcal F})=\Omega ^p(Q)/\Omega ^p(Q,\mathcal F).$$
The operator $d_p:\Omega^p({\mathcal F})\rightarrow \Omega
^{p+1}(\mathcal F),\;d_p(\omega )=d\omega$ is well defined since
$d\Omega ^p(Q,{\mathcal F})\subset \Omega ^{p+1}(Q,{\mathcal F})$.
Therefore one has the following differential complex
$$
(\Omega ^{\ast}({\mathcal F}),d):\Omega ^0({\mathcal
F})\rightarrow ^{d_0}\Omega ^1({\mathcal F})\rightarrow
^{d_1}\Omega ^2({\mathcal F})\rightarrow ^{d_2}\ldots
$$
The cohomology  of this complex
$$
H^p({\mathcal F})=H^p(\Omega^p({\mathcal F}),d)=\mbox{ker}
d_p/\mbox{im} d_{p-1}
$$
is called {\bf de Rham cohomology tangent to foliation} ${\mathcal
F}$ (see \cite{Vaisman2} for a similar construction).

Directly from the definition of the cohomology $H^p({\mathcal F})$
we get the following propositions.
\begin{pr}
$H^p({\mathcal F})=0$ for $p$ greater than the dimension of the
leaves of the foliation $\mathcal F$.
\end{pr}

\begin{pr}
Let $\omega\in \Omega^p(Q)$ such that $d\omega \in
\Omega^{p+1}(Q,{\mathcal F})$. If $[\omega]=0$ in $H^p({\mathcal
F})$ then $[\iota_q^{\ast}\omega]=0$ in $H^p({\mathcal F}_q)$ for
any $q\in Q$.
\end{pr}

We define the special kind of contractions along a foliation.
\begin{df}
We say that $Q$ is contractible to a submanifold $S$ along the
foliation $\mathcal F$ if there exist a family of maps
$F_t:Q\rightarrow Q$, $t\in [0,1]$ which is (piece-wise) smooth in
$t$, such that $F_1$ is the identity map, $F_0(Q)\subset
S,\;F_0|_S=id_S$ and $\forall q\in Q\;F_t({\mathcal F}_q)\subset
{\mathcal F}_q$ for all $t\in [0,1]$. We call the family $F_t$ a
(piece-wise) smooth contraction of $Q$ to $S$ along the foliation
$\mathcal F$.
\end{df}

Using the analog of the homotopy operator for the above
contraction we prove the following theorem.
\begin{thm}\label{contr}
Let $S$ be a smooth submanifold of $Q$ transversal to a foliation
$\mathcal F$. If $Q$ is contractible to $S$ along the foliation
$\mathcal F$ then the cohomology groups $H^p({\mathcal F})$ and
$H^p({\mathcal F}\cap S)$ are isomorphic.
\end{thm}

\begin{proof} Since $\mathcal F$ is transversal to S then ${\mathcal
F}\cap S$ is a foliation on $S$. Let $F_t$ be a (piece-wise)
smooth contraction of $Q$ to $S$ along the foliation $\mathcal F$.
Let $\omega$ be a $p$-form on $Q$ such that $d\omega \in
\Omega^{p+1}(Q,{\mathcal F})$ and $i:S\hookrightarrow Q$ be the
standard inclusion of $S$ in $Q$. Then
$$
\omega -(F_0^{\ast }\circ \iota^{\ast})(\omega )=F_1^{\ast }\omega
-F_0^{\ast }\omega = \int _0^1\frac{d}{dt}F_t^{\ast }\omega
dt=\int _0^1F_t^{\ast }({\mathcal L}_{V_t}\omega )dt=
$$
$$
\int _0^1F_t^{\ast }(V_t\rfloor d\omega +d(V_t\rfloor \omega
))dt=\int _0^1[F_t^{\ast }(V_t\rfloor d\omega )+d(F_t^{\ast
}(V_t\rfloor \omega ))]dt
$$
where $V_t\circ F_t=\frac{dF_t}{dt}$ and ${\mathcal L}_{V_t}$ is
Lie derivative along the vector field $V_t$.

For any $q\in Q$ and any $(u_1,\cdots,u_p)$ tangent to ${\mathcal
F}_q$ we have
$$
\int _0^1F_t^{\ast }(V_t\rfloor d\omega )dt(u_1,\cdots,u_p)=\int
_0^1F_t^{\ast }(V_t\rfloor d\omega )(u_1,\ldots ,u_p)dt=
$$
$$
=\int _0^1d\omega (V_t\circ F_t,F_{t_{\ast }}u_1,\ldots
,F_{t_{\ast }}u_p)dt=\int _0^1d\omega (\frac{dF_t}{dt},F_{t_{\ast
}}u_1,\ldots ,F_{t_{\ast }}u_p)dt=0,
$$
since $F_t({\mathcal F}_q)={\mathcal F}_q$ and $d\omega\in
\Omega^{p+1}(Q,{\mathcal F})$. It implies that $\int _0^1F_t^{\ast
}(V_t\rfloor d\omega )dt\in \Omega^p(Q,{\mathcal F})$

Finally we obtain,
$$
\omega -(F_0^{\ast }\circ \iota^{\ast})(\omega )=\beta+d\alpha,
$$
where $\beta \in \Omega^p(Q,{\mathcal F})$ and $\alpha =\int
_0^1F_t^{\ast }(V_t\rfloor \omega )dt$. Thus
$$
[\omega ]=[(F_0^{\ast }\iota^{\ast}(\omega)]\in H^p({\mathcal F})
$$
It implies that $F_0^{\ast }\circ \iota^{\ast}=id_{H^p({\mathcal
F})}$.

On the other hand, $\iota^{\ast}\circ F_0^{\ast
}=id_{H^p({\mathcal F}\cap S)}$, since $F_0\circ \iota=id_S$. Thus
$F_0^{\ast}$ is the required isomorphism between cohomology groups
$H^p({\mathcal F}\cap S)$ and $H^p({\mathcal F})$.\end{proof}

\section{Integrability of characteristic distribution.}

Let $Q$ be a submanifold of a locally conformal symplectic
manifold $(M,\Omega,\omega)$, $\dim M=2n$. Let $\iota:
Q\hookrightarrow M$ denote the standard inclusion of $Q$ in $M$.
Let
$$
(T_qQ)^{\Omega}=\left\{v\in T_qM | \Omega(v,w)=0 \ \forall w\in
T_qQ \right\}.
$$
We assume that $\dim (T_qQ)^{\Omega}\cap T_qQ$ is constant for
every $q\in Q$ . By $(TQ)^{\Omega}\cap TQ$ we denote a
characteristic distribution $\bigcup_{q\in Q} (T_qQ)^{\Omega}\cap
T_qQ $ which is a subbundle of the tangent bundle to $Q$. Now we
prove
\begin{pr}
\label{invol} The characteristic distribution $(TQ)^{\Omega}\cap
TQ$ is involutive.
\end{pr}

\begin{proof} Let $X$, $Y$ be smooth sections of $(TQ)^{\Omega}\cap TQ$.
We show that $[X,Y]$ is also a section of $(TQ)^{\Omega}\cap TQ$.
By the well-known formula, for a smooth vector field $Z$ on $Q$ we
have
\begin{eqnarray*}
&d\Omega(X,Y,Z)&=\\
 & X(\Omega(Y,Z))-Y(\Omega(X,Z))+Z(\Omega(X,Y))+& \\
& +\Omega([X,Z],Y)-\Omega([X,Y],Z)-\Omega([Y,Z],X)& = \\
&-\Omega([X,Y],Z),&
\end{eqnarray*}
because $X, Y$ are smooth sections of $(TQ)^{\Omega}\cap TQ$. On
the other hand $d\Omega=\omega\wedge\Omega$. Therefore
$$
d\Omega(X,Y,Z)=\omega(X)\Omega(Y,Z)+\omega(Y)\Omega(Z,X)+\omega(Z)\Omega(X,Y)=0.
$$
Thus we obtain $\Omega([X,Y],Z)=0$ for every smooth vector field
$Z$ on $Q$. On the other hand $[X,Y]$ is a section of $TQ$, since
$X,Y$ are sections of $TQ$. Therefore $[X,Y]$ is a smooth section
of $(TQ)^{\Omega}\cap TQ$.\end{proof}

\section{Reduction of locally conformal symplectic
forms.}\label{form}

By Frobenius' theorem  and Proposition \ref{invol},
$(TQ)^{\Omega}\cap TQ$ is integrable and defines a foliation
${\mathcal F}$, which is called a characteristic foliation. Let
$N=Q/{\mathcal F}$ be a quotient space obtained by identification
of all points on a leaf. Assume that $N=Q/{\mathcal F}$ is a
smooth manifold and the canonical projection $\pi : Q \rightarrow
N=Q/{\mathcal F}$ is a submersion.  If $\Omega$ is a symplectic
form then there exists a symplectic structure $\tau$ on $N$ such
that
\begin{equation}
\label{war1} \pi^{\ast}\tau=\iota^{\ast}\Omega,
\end{equation}
where $\iota: Q\hookrightarrow M$ denotes the standard inclusion
of $Q$ in $M$ (see \cite{MW}, \cite{A-M}, \cite{Arnold},
\cite{GS}, \cite{MR}, \cite{O-R}, \cite{D-J1} and many others).

In \cite{H-R3} the reduction procedure for locally conformal
symplectic manifolds that satisfies condition (\ref{war1}) is
proposed.

The necessary and sufficient condition for existence of a
conformal symplectic form on the reduced manifold $N=Q/{\mathcal
F}$, which satisfies condition (\ref{war1}) is presented in the
following theorem (see also Section 3 in \cite{H-R3}).

\begin{pr}
\label{tw-ex} Let $Q$ be a submanifold of a locally conformal
symplectic structure $(M,\Omega,\omega)$, let $\iota:
Q\hookrightarrow M$ denote the standard inclusion of $Q$ in $M$
and let $\mathcal F$ be the characteristic foliation of the
characteristic distribution $TQ^{\Omega}\cap TQ$ of constant
dimension. If $N=Q/{\mathcal F}$ is a manifold of dimension
greater than 2 and the canonical projection $\pi: Q \rightarrow
N=Q/{\mathcal F}$ is a submersion then there exists a locally
conformal symplectic form $\tau$ on $N$ such that
$\pi^{\ast}\tau=\iota^{\ast}\Omega$ if and only if
\begin{equation}
\label{war1a} \iota^{\ast}\omega(X)=0
\end{equation} for every smooth section $X$ of $TQ^{\Omega}\cap TQ$.
\end{pr}

 \begin{proof} Assume that there exists a locally conformal symplectic form $\tau$ on $N$
 such that
 $\pi^{\ast}\tau=\iota^{\ast}\Omega$. Then
$$
\pi^{\ast}d\tau=\iota^{\ast}d\Omega=\iota^{\ast}\omega\wedge
\iota^{\ast}\Omega.
$$
Therefore, for every smooth section $X$ of $TQ^{\Omega}\cap TQ$ we
have
$$
X\rfloor (\iota^{\ast} \omega \wedge \iota^{\ast}\Omega)=X\rfloor
(\pi^{\ast}d\tau)=0,
$$
because $\pi_{\ast}(X)=0$. But $\iota^{\ast}\Omega\ne 0$ and
$X\rfloor \iota^{\ast}\Omega=0$ , therefore
$\iota^{\ast}\omega(X)=0$.

Now assume that $\iota^{\ast}\omega(X)=0$ for every smooth section
$X$ of $TQ^{\Omega}\cap TQ$. Then
$$
 X\rfloor d\iota^{\ast}\Omega=X\rfloor (\iota^{\ast} \omega \wedge \iota^{\ast}\Omega)=0.
$$
Hence
$$
L_X\iota^{\ast}\Omega=X\rfloor (d\iota^{\ast}\Omega)+d(X\rfloor
\iota^{\ast}\Omega)=0
$$
for every smooth section $X$ of $TQ^{\Omega}\cap TQ$. Therefore
$\Omega$ is constant on every leaf of the characteristic foliation
${\mathcal F}$. Now we define the form $\tau$ by the formula
$$
\pi^{\ast}\tau=\iota^{\ast}\Omega.
$$
$\tau$ is well-defined, because $\pi$ is a submersion. It is
nondegenerate, because the kernel of $\iota^{\ast}\Omega$ is
$TQ^{\Omega}\cap TQ=ker \pi_{\ast}$. From the definition of $\tau$
we obtain
\begin{equation}
\label{d-conf} \pi^{\ast}d\tau=d\iota^{\ast}\Omega=\iota^{\ast}
\omega \wedge \iota^{\ast}\Omega= \iota^{\ast} \omega \wedge
\pi^{\ast}\tau.
\end{equation}
We define $\alpha$ by the formula
$$
\pi^{\ast}\alpha=\iota^{\ast}\omega.
$$
$\alpha$ is well-defined closed $1-$form on $N$, because $\pi$ is
a submersion and $\omega$ is closed. From (\ref{d-conf}) we have
$d\tau=\alpha \wedge \tau$.\end{proof}

Notice that a generic hypersurface on $M$ does not satisfy (even
locally) assumption (\ref{war1a}).

\begin{ex}
\label{p1} Let $H$ be a smooth hypersurface on a locally conformal
symplectic manifold  $(M,\Omega,\omega)$.

By Darboux theorem germs at $q$ of $(M,\Omega,\omega)$ and $H$ are
locally diffeomorphic to germs at $0$ of $(\mathbb
R^{2n},f\sum_{i=1}^n dx_i\wedge dy_i,df)$ and $\{(x,y)\in \mathbb
R^{2n}:x_1=0\}$, where $f$ is a smooth function-germ on $\mathbb
R$ at $0$ and $\dim M=2n$.

Then the characteristic distribution $T\{(x,y)\in \mathbb
R^{2n}:x_1=0\}^{\Omega}$ is spanned by $\frac{\partial}{\partial
y_1}$. The reduced manifold can be locally  identified with
$\{(x,y)\in \mathbb R^{2n}:x_1=y_1=0\}$.

There exists a locally conformal symplectic structure $\tau$ on
the reduced manifold satisfying condition (\ref{war1}) if and only
if $\frac{\partial f}{\partial y_1}|_{\{x_1=0\}}=0$.
\end{ex}

 In the next section we
propose a procedure of reduction of locally conformal symplectic
structures and find the sufficient and necessary condition for
this reduction in terms of a cohomology class of the restriction
of $\omega$ to the coisotropic submanifold in the first cohomology
group tangent to its characteristic foliation.

\section{Reduction of locally conformal symplectic structures.}

Let $(M,\Omega,\omega)$ be a locally conformal symplectic
manifold. Let $Q$ be submanifold of $M$, let $\iota:
Q\hookrightarrow M$ denote the standard inclusion of $Q$ in $M$
and let $\mathcal F$ be the characteristic foliation of the
characteristic distribution $TQ^{\Omega}\cap TQ$ of constant
dimension smaller than $\dim Q$.

\begin{pr} \label{eq}
If $\Omega^{\prime}$ is a locally conformal symplectic form
conformally equivalent  to $\Omega$ on $Q$ then the characteristic
foliation $\mathcal F^{\prime}$ of $TQ^{\Omega^{\prime}}\cap TQ$
coincides with $\mathcal F$. If $\omega^{\prime}$ is the Lee form
of $\Omega^{\prime}$ then
$[\iota^{\ast}\omega]=[\iota^{\ast}\omega^{\prime}]$ in
$H^1({\mathcal F})$.
\end{pr}
\begin{proof} $\Omega$ and $\Omega^{\prime}$  are conformally equivalent
on $Q$ then there exists a positive function $f$ on $Q$ such that
\begin{equation}
\label{conQ} \iota^{\ast}\Omega=f\iota^{\ast}\Omega^{\prime}.
\end{equation}
Thus it is obvious that ${\mathcal F}={\mathcal F^{\prime}}$,
since $TQ^{\Omega^{\prime}}\cap TQ=ker
\iota^{\ast}\Omega^{\prime}=ker \iota^{\ast}\Omega=TQ^{\Omega}\cap
TQ$. Differentiating (\ref{conQ}) we obtain
$$
\iota^{\ast}\omega\wedge\iota^{\ast}\Omega=f\iota^{\ast}\omega^{\prime}\wedge\iota^{\ast}\Omega^{\prime}
+df\wedge\iota^{\ast}\Omega^{\prime}
$$
Using (\ref{conQ}) again we have
$$
(\iota^{\ast}\omega-\iota^{\ast}\omega^{\prime}-d(\ln
f))\wedge\iota^{\ast}\Omega^{\prime}=0
$$
Let $v$ be a vector tangent to a foliation. Then
$$
v\rfloor(\iota^{\ast}\omega-\iota^{\ast}\omega^{\prime}-d(\ln
f))\wedge\iota^{\ast}\Omega^{\prime}=0.
$$
But $\iota^{\ast}\Omega^{\prime}\ne 0$ and $v\rfloor
\iota^{\ast}\Omega^{\prime}=0$, therefore
$$\iota^{\ast}\omega(v)-\iota^{\ast}\omega^{\prime}(v)-
d(\ln(f))(v)=0.$$
 It implies that
$[\iota^{\ast}\omega]=[\iota^{\ast}\omega^{\prime}]$ in
$H^1({\mathcal F})$.\end{proof}

Proposition \ref{eq} means that the cohomology class
$[\iota^{\ast}\omega]$ in $H^1({\mathcal F})$ is the invariant of
the restriction of a locally conformal symplectic structure to
$Q$. In the next theorem we use this class to state the necessary
and sufficient condition when a reduced locally conformal
symplectic structure exists on a reduced manifold.
\begin{thm}
\label{tw-coh} Let $Q$ be a submanifold of a locally conformal
symplectic manifold $(M,\Omega,\omega)$, let $\iota:
Q\hookrightarrow M$ denote the standard inclusion of $Q$ in $M$
and let $\mathcal F$ be the characteristic foliation of the
characteristic distribution $TQ^{\Omega}\cap TQ$ of constant
dimension.

If $N=Q/\mathcal F$ is a manifold of dimension greater than 2 and
the canonical projection $\pi: Q \rightarrow N=Q/{\mathcal F}$ is
a submersion then there exists a locally conformal symplectic form
$\tau$ on $N$ and a smooth positive function $f$ on $Q$ such that
\begin{equation}
\label{war2} \pi^{\ast}\tau=f\iota^{\ast}\Omega
\end{equation} if
and only if $[\iota^{\ast}\omega]=0 \in H^1({\mathcal F})$
\end{thm}

\begin{proof} Assume that there exists a locally conformal symplectic
form $\tau$ on $N$ and a positive smooth function $f$ on $Q$
 such that
 $\pi^{\ast}\tau= f \iota^{\ast}\Omega$. Then
$$
\pi^{\ast}d\tau= df\wedge \iota^{\ast}\Omega+
f\iota^{\ast}d\Omega= f(d(\ln(f))+\iota^{\ast}\omega)\wedge
\iota^{\ast}\Omega.
$$
Therefore, for any $q\in Q$ and any vector $v$ tangent to
${\mathcal F}_q$ we have
$$
v\rfloor (\iota^{\ast} \omega+d(\ln(f))) \wedge
\iota^{\ast}\Omega)= v\rfloor (\frac{\pi^{\ast}d\tau}{f})=0,
$$
since $\pi_{\ast}(v)=0$. But $\iota^{\ast}\Omega\ne 0$ and
$v\rfloor \iota^{\ast}\Omega=0$, therefore
$$\iota^{\ast}\omega(v)+ d(\ln(f))(v)=0$$ for any $v$ tangent to
${\mathcal F}_q$. It implies that $[\iota^{\ast}\omega]=0\in
H^1({\mathcal F})$.

Now assume that $[\iota^{\ast}\omega]=0\in H^1({\mathcal F})$.
Then there exists a function $g$ on $Q$ such that
$\iota^{\ast}\omega(v)=dg(v)$ for any $q\in Q$ and any vector $v$
tangent to ${\mathcal F}_q$. Thus
$$
 v\rfloor d(\exp(-g)\iota^{\ast}\Omega)=
 $$
 $$
 \exp(-g)(-dg(v)+\iota^{\ast} \omega(v)) \iota^{\ast}\Omega -
 \exp(-g)(-dg+\iota^{\ast} \omega) \wedge v \rfloor \iota^{\ast}\Omega=0.
$$
Hence
$$
{\mathcal L}_X\exp(-g)\iota^{\ast}\Omega=X\rfloor
d(\exp(-g)\iota^{\ast}\Omega)+d(X\rfloor
\exp(-g)\iota^{\ast}\Omega)=0
$$
for every smooth section $X$ of $TQ^{\Omega}\cap TQ$. Therefore
$\exp(-g)\iota^{\ast}\Omega$ is constant on every leaf of the
characteristic foliation ${\mathcal F}$. Now we define the form
$\tau$ by the formula
$$
\pi^{\ast}\tau=\exp(-g)\iota^{\ast}\Omega.
$$
$\tau$ is well-defined, because $\pi$ is a submersion. It is
nondegenerate, because the kernel of $\exp(-g)\iota^{\ast}\Omega$
is $TQ^{\Omega}\cap TQ=ker \pi_{\ast}$. From the definition of
$\tau$ we obtain
\begin{equation}
\label{d-conf1}
\pi^{\ast}d\tau=d(exp(-g)\iota^{\ast}\Omega)=(\iota^{\ast}
\omega-dg) \wedge \exp(-g)\iota^{\ast}\Omega= (\iota^{\ast}
\omega-dg) \wedge \pi^{\ast}\tau.
\end{equation}
We define $\alpha$ by the formula
$$
\pi^{\ast}\alpha=\iota^{\ast}\omega-dg.
$$
$\alpha$ is well-defined closed $1-$form on $N$, because $\pi$ is
a submersion and $\omega$ is closed. From (\ref{d-conf1}) we have
$d\tau=\alpha \wedge \tau$.\end{proof}

By Proposition \ref{eq} and Theorem \ref{tw-coh} it is easy to see
that the reduction does not depend of the choice of a locally
conformal symplectic form from the conformal equivalence class.
Two locally conformal symplectic forms conformally equivalent on
$Q$ are reduced to the same locally conformal symplectic structure
on a reduced space. Thus Theorem \ref{tw-coh} gives a procedure of
reduction of locally conformal symplectic structures.

Now we show how this procedure of reduction works. Any germ of a
coisotropic submanifold can be reduced using the above procedure,
since a locally conformal symplectic manifold is locally
equivalent to a symplectic manifold. The obstruction to existence
of the locally conformal structure on the reduced manifold is only
global.

\begin{cor}
Let $Q$ be the germ at $q$ of a coisotropic submanifold of a
locally conformal symplectic manifold $(M,\Omega,\omega)$, let
$\iota: Q\hookrightarrow M$ denote the germ of the inclusion of
$Q$ in $M$ and let $\mathcal F$ be the characteristic foliation of
$TQ^{\Omega}\cap TQ=TQ^{\Omega}$.

 Then there exists a
germ locally conformal symplectic form $\tau$ on the germ of the
reduced manifold $N=Q/\mathcal F$ and a germ of a smooth positive
function $f$ on $Q$ such that $\pi^{\ast}\tau=
f\iota^{\ast}\Omega$ where $\pi: Q \rightarrow N=Q/{\mathcal F}$
is the germ of the canonical projection.
\end{cor}

\end{document}